\documentclass[11pt,letterpaper]{article}

\usepackage{amsmath,amsfonts,amssymb,color}
\usepackage{amsthm}
\usepackage[margin=1in]{geometry}
\usepackage{graphics} 
\usepackage{tikz,,subcaption,cite}
\usetikzlibrary{positioning}
\usetikzlibrary{arrows}
\usepackage{natbib} 
\usepackage{epstopdf}
\usepackage{epsfig} 
\usepackage{psfrag}
\usepackage[width=.85\textwidth]{caption}
\usepackage{lipsum}
\usepackage{algorithm} 
\usepackage{algorithmic}
\usepackage[title,titletoc,toc]{appendix}
\usepackage{url}
\usepackage[width=.85\textwidth, font=footnotesize]{caption}
\usepackage[pdfpagelabels,pdfpagemode=None,breaklinks=true]{hyperref} \hypersetup{
     colorlinks   = true,
     citecolor    = blue
}
\renewcommand{\cite}{\citep}

\newtheorem{theorem}{Theorem}
\newtheorem{lemma}{Lemma}
\newtheorem{corollary}{Corollary}

\newtheorem{proposition}{Proposition}

\newtheorem{definition}{Definition}

\newtheorem{problem}{Problem}

\makeatletter
\renewcommand\bibsection%
{
  \section*{\refname
    \@mkboth{\MakeUppercase{\refname}}{\MakeUppercase{\refname}}}
}
\makeatother

\newcommand{\ignore}[1]{}

\DeclareMathAlphabet{\mathcal}{OMS}{cmsy}{m}{n}      
      
\title{\LARGE \bf Pursuit on a Graph under Partial Information\\ from Sensors}
\author{Shreyas Sundaram,\thanks{S. Sundaram is with the School of Electrical and Computer Engineering at Purdue University, W. Lafayette, IN 47907. E-mail: {\tt sundara2@purdue.edu}. S. Sundaram's work was supported by the U. S. Air Force Research Lab Summer Faculty Fellowship Program (AFRL-SFFP).}~ Krishnamoorthy Kalyanam\thanks{K. Krishnamoorthy is with the InfoSciTex corporation (AFRL Contractor), Dayton, OH 45431.  E-mail: {\tt krishnak@ucla.edu}.}  ~and~David W. Casbeer\thanks{D. Casbeer is with the Autonomous Control Branch, Air Force Research Lab, Wright-Patterson AFB, OH 45433.  E-mail: {\tt david.casbeer@us.af.mil}.}
}

\begin{document}

\date{}
\maketitle
\thispagestyle{empty}
\pagestyle{empty}

\begin{abstract}
We consider a class of pursuit-evasion problems where an evader enters a directed acyclic graph and attempts to reach one of the terminal nodes.  A pursuer enters the graph at a later time and attempts to capture the evader before it reaches a terminal node.  The pursuer can only obtain information about the evader's path via sensors located at each node in the graph; the sensor measurements are either green or red (indicating whether or not the evader has passed through that node).  We first show that it is NP-hard to determine whether the pursuer can enter with some nonzero delay and still be guaranteed to capture the evader, even for the simplest case when the underlying graph is a tree.  This also implies that it is NP-hard to determine the largest delay at which the pursuer can enter and still have a guaranteed capture policy. We further show that it is NP-hard to approximate (within any constant factor) the largest delay at which the pursuer can enter. Finally, we provide an algorithm to compute the maximum pursuer delay for a class of node-sweeping policies on tree networks and show that this algorithm runs in linear-time for bounded-degree trees.
\end{abstract}

\section{Introduction}
\label{sec:intro}

The problem of capturing an evader (or target) by one or more pursuers has a long history in computer science, discrete mathematics, differential game theory, and control theory \cite{parsons1978pursuit, isaacs1999differential, aigner1984game}, covering a large variety of different formulations.   For example,  in {\it Cops-and-Robbers} games, multiple pursuers seek to capture an evader, often under the assumption of full visibility of the evader \cite{aigner1984game, bonato2011game, seymour1993graph, fomin2008annotated,megiddo1988complexity}.   

In this paper, we consider a class of pursuit-evasion problem on {\it graphs} under {\it partial information} for the pursuer.  Specifically, we consider a directed acyclic graph where the evader enters at a source node and attempts to reach a terminal node.  The pursuer enters at some later time and attempts to capture the evader at one of the nodes of the graph before it reaches its target.  However, the pursuer can {\it only} obtain information about the evader's location by visiting sensors located at the nodes of the graph; these sensors measure whether or not the evader passes through the node.  This scenario, where the information available to the pursuer is a function of the actions taken by the pursuer, differs from other related work on pursuit-evasion with partial information \cite{chung2011search}.  For example, in   \cite{clarke2009witness, clarke2006cops}, the pursuers obtain information about the evader via witnesses or alarms, regardless of their location in the graph.  Similarly, in \cite{isler2008role,johnson1983column}, the pursuer can sense the evader only when they are sufficiently close together.  In contrast, our work considers the case where the pursuer must explicitly visit certain locations of the graph in order to gain information about the evader.

Previous work that has studied the same general setting as ours includes \cite{krishnamoorthy2013optimal} where the underlying graph is a Manhattan grid, \cite{chen2016intruder} where a dynamic programming approach was provided to analyze general networks, and \cite{krishnamoorthy2016pursuit} which provided an (exponential-time) algorithm to calculate the pursuer policy that guarantees capture while maximizing the pursuer entrance delay. The contributions of this paper are as follows.  First, we provide a formal complexity characterization of finding optimal pursuer policies for this class of pursuit-evasion problems;  specifically, we show that it is NP-hard to determine whether the pursuer can guarantee capture after entering with a positive delay.  Second, we show that it is NP-hard to approximate the maximum pursuer entrance delay within any finite constant factor.  
Third, we provide an explicit algorithm to calculate the maximum pursuer delay for a specific class of {\it node-sweeping} policies (to be precisely defined later) on tree networks, and show that this algorithm runs in linear-time for bounded-degree trees.  Our results make connections to traveling salesperson and vehicle routing problems with time-windows, thereby providing new insights  into this class of pursuit-evasion problems on graphs under partial information.
\section{The Pursuit-Evasion Problem}
\label{sec:problem}

Consider a directed acyclic graph (DAG) $\mathcal{R} = \{\mathcal{V}_R, \mathcal{E}_R\}$, representing a road network.\footnote{We adopt  standard graph-theoretic terminology throughout (e.g., see \cite{cormen2009introduction}).} The graph has a single {\it source} (or {\it root}) node $s \in \mathcal{V}_R$.  The nodes that have no outgoing edges are called {\it goal nodes} and denoted by the set $\mathcal{G} \subset \mathcal{V}_R$.  

A ground vehicle (the evader) enters the network through the source node $s$ at time $t = 0$.  The time taken by the evader to go from $v_i$ to $v_j$ (if that edge exists in the network) is denoted by $d_e(v_i,v_j)$.  The objective of the evader is to reach one of the goal nodes in $\mathcal{G}$.  

Each node $v_i \in \mathcal{V}_R$ has an Unattended Ground Sensor (UGS).  The sensor can either be in state ``green'' indicating that the evader has not yet passed through the node containing that sensor, or in state ``red'' with an associated time-stamp indicating when the evader passed through that node.

There is an unmanned aerial vehicle (the pursuer) which enters the road network via the source node $s$ at time $t = D$, for some $D \ge 0$.  We refer to the pursuer entrance time $D$ as the {\it pursuer delay}.   The pursuer can move between any pair of nodes in the network, with a travel time of $d_p(v_i,v_j)$ for going from $v_i$ to $v_j$.   These pursuer travel times are symmetric, nonnegative, and satisfy the triangle inequality
$$
d_p(v_i, v_j) \le d_p(v_i,v_k) + d_p(v_k,v_j) \enspace \forall v_i, v_k, v_j \in \mathcal{V}_R.
$$
Furthermore, the pursuer has a speed advantage over the evader, i.e.,  $d_p(v_i, v_j) \le d_e(v_i,v_j)$ for all $(v_i,v_j) \in \mathcal{E}_R$.  

The objective of the pursuer is to capture the evader before it reaches a goal node.  The pursuer can only obtain information about the evader's movements via the UGSs.  Specifically, when the pursuer reaches a node $v_i \in \mathcal{V}_R$, it obtains the state (green or red) of the UGS, and if red, the time at which the evader passed through that node.  After arriving at a node and obtaining the UGS measurement, the pursuer can decide which node to move to next (or stay at the current node).  The pursuer captures the evader if and only if it is at a node at the time the evader reaches that node.  

At any given point in time $t$, let $\mathcal{H}(t)$ be the history of the nodes visited by the pursuer up to time $t$, along with the measurements received from the corresponding UGSs.  A {\it policy} for the pursuer is a mapping $\mu$ from the history $\mathcal{H}(t)$ and the pursuer's current node  to the next node that the pursuer should visit.

The above model is summarized as follows.\footnote{Since we will be interested in problems that have finite representations, we will henceforth take all distances to be  nonnegative integers.}

\begin{definition} 
An instance of the Pursuit-Evasion problem is given by a DAG $\mathcal{R} = \{\mathcal{V}_R,\mathcal{E}_R\}$ containing a single source node $s$, a nonnegative evader travel time $d_e(v_i,v_j)$ for each edge $(v_i,v_j) \in \mathcal{E}_R$, and a nonnegative pursuer travel time $d_p(v_i,v_j)$ for all distinct pairs of vertices $v_i, v_j \in \mathcal{V}_R$.  The pursuer travel times are symmetric, satisfy the triangle inequality,  and $d_p(v_i,v_j) \le d_e(v_i,v_j)$ for all $(v_i,v_j) \in \mathcal{E}_R$.
\label{def:instance}
\end{definition}
  
Note that capture is always guaranteed if $D = 0$ (as the pursuer and evader will be co-located at the source node in that case).  We will be considering the following objectives within the above class of Pursuit-Evasion problems. 

\begin{problem}
{\it Maximum Pursuer Delay Problem (MPDP).} Given an instance of the  Pursuit-Evasion problem, find the largest time $D^*$ at which the pursuer can enter the graph so that there is a policy that guarantees capture of the evader.
\label{prob:max_delay}
\end{problem}

\begin{problem}
{\it Capture Feasibility Problem (CFP).} Given an instance of the Pursuit-Evasion problem, is there some time $D > 0$ at which the pursuer can enter the graph and still be guaranteed to capture the evader?  
\label{prob:capture_feasibility}
\end{problem}

Note that an algorithm that solves the MPDP will also yield an answer to the CFP.  We will show that the CFP is NP-hard, which then implies NP-hardness of MPDP as well.  In the next section, we characterize the solution of the CFP and MPDP for a specific class of instances of the Pursuit-Evasion problem, which will subsequently lead to the results described above.

\section{Pursuit-Evasion on Spider Networks}
\label{sec:spider_network}

Consider a class of road networks of the following form.  Let $r \in \mathbb{N}$.  The node set is partitioned as $\mathcal{V}_R = \{s\} \cup \mathcal{C} \cup \mathcal{G}$, where $\mathcal{C} = \{v_{c_1}, v_{c_2}, \ldots, v_{c_r}\}$ is a set of {\it core nodes}, and $\mathcal{G} = \{v_{g_1}, v_{g_2}, \ldots, v_{g_r}\}$ is the set of goal nodes.   The edge set is defined as
$$
\mathcal{E}_R = \left\{(s, v_{c_i}), 1 \le i \le r\right\} \cup \left\{(v_{c_i},v_{g_i}), 1 \le i \le r\right\}.
$$

The travel times for the pursuer and evader on this graph are defined as follows.  For the pursuer, let the distance function $d_p(\cdot, \cdot)$ be positive and satisfy the triangle inequality, but otherwise arbitrary.  For the evader, let each edge $(s,v_{c_i}) \in \mathcal{E}_R$ have length $d_e(s,v_{c_i}) = d_p(s, v_{c_i})$ (i.e., the time taken for the evader to go from the source node to a core node is the same as the corresponding travel time for the pursuer).  For each edge $(v_{c_i}, v_{g_i})$, define the length to be  $d_e(v_{c_i}, v_{g_i}) = L - d_e(s, v_{c_i})$ for some $L \in \mathbb{Z}_{> 0}$ satisfying $L - d_e(s, v_{c_i}) > d_p(v_{c_i},v_{g_i})$.   Thus, regardless of the path taken by the evader, it will arrive at the corresponding goal node at time $L$.  We refer to the above road network as a {\it spider network} (see Fig.~\ref{fig:spider_network} for an illustration).

\begin{figure}[t]
\begin{center}
\begin{tikzpicture}[scale = 0.25]
\def\nodesize{20pt}

  \node [circle, inner sep=0pt, minimum size=\nodesize, draw] (s) at (-11,0)  {\footnotesize $s$};

   \node [circle, inner sep=0pt, minimum size=\nodesize, draw] (vc1) at (0.5,5.5)  {\footnotesize $v_{c_1}$};
   \node [circle, inner sep=0pt, minimum size=\nodesize, draw] (vc2) at (0.5,0)  {\footnotesize $v_{c_2}$};
   \node [circle, inner sep=0pt, minimum size=\nodesize, draw] (vc3) at (0.5,-5.5)  {\footnotesize $v_{c_3}$};
   
   \draw[rounded corners, dashed] (-1.25, -7.5) rectangle (2.25, 7.5) {};
   \node at (0.5, 8.5) {$\mathcal{C}$};
  
  \draw [thick, ->] (s) -- node[sloped, above ] {\footnotesize $d_e(s,v_{c_1})$}  ++(vc1);
  \draw [thick, ->] (s) -- node[pos=0.65, above] {\footnotesize $d_e(s,v_{c_2})$}  ++(vc2);
  \draw [thick, ->] (s) -- node[sloped, below] {\footnotesize $d_e(s,v_{c_3})$}  ++(vc3);


  \node [circle, inner sep=0pt, minimum size=\nodesize, draw] (vg1) at (12,5.5) {\footnotesize $v_{g_1}$};
  \node [circle, inner sep=0pt, minimum size=\nodesize, draw] (vg2) at (12,0) {\footnotesize $v_{g_2}$};
  \node [circle, inner sep=0pt, minimum size=\nodesize, draw] (vg3) at (12,-5.5) {\footnotesize $v_{g_3}$};
  
   \draw[rounded corners, dashed] (10.25, -7.5) rectangle (13.75, 7.5) {};
   \node at (12, 8.5) {$\mathcal{G}$};
  
  
  \draw [thick, ->] (vc1) -- node[above] {\footnotesize $L-d_e(s,v_{c_1})$}  ++(vg1);
  \draw [thick, ->] (vc2) -- node[above] {\footnotesize $L-d_e(s,v_{c_2})$}  ++(vg2);
  \draw [thick, ->] (vc3) -- node[above] {\footnotesize $L-d_e(s,v_{c_3})$}  ++(vg3);


  
\end{tikzpicture}
\caption{A spider network with $r =3$.  The edge labels indicate the travel times for the evader.}
\label{fig:spider_network}
\end{center}
\end{figure}
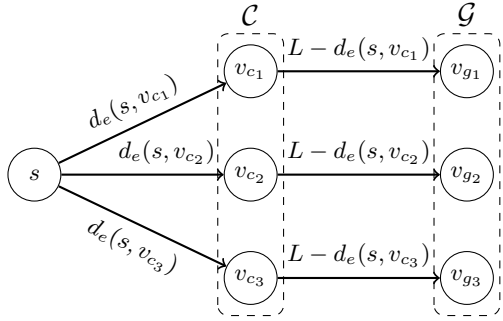

Now suppose the pursuer enters the network at time $D > 0$.  Since the travel times to go from the source to a core node are identical for the pursuer and the evader, and since the travel times satisfy the triangle inequality, the evader is guaranteed to have passed through one of the core nodes in $\mathcal{C}$ by the time the pursuer reaches any core node.  We thus have the following fact about the pursuer's optimal policy.

 \begin{proposition}
 Consider the Pursuit-Evasion problem on a spider network, where the pursuer enters at time $D > 0$.  Suppose there exists a pursuer policy $\mu^{\ast}$ that guarantees capture of the evader.  Then, there exists a policy $\mu$ (perhaps the same as $\mu^*$) that also guarantees capture of the evader and has the following property.  At each time $t > 0$,  $\mu(\mathcal{H}(t)) \in \mathcal{G}$ if and only if one of the following two conditions hold:
 \begin{enumerate}
 \item The pursuer has visited a core node with a red state at or prior to time $t$.
 \item The pursuer has visited $r-1$ of the core nodes at or prior to time $t$, all of which were in a green state.
 \end{enumerate}
 \label{prop:optimal_capture_goal_node}
 \end{proposition} 

\begin{proof}
We will start with the policy $\mu^*$ and modify it to yield a policy $\mu'$ that satisfies the ``if'' part of the proposition;  we will then further modify $\mu'$ to obtain a policy $\mu$ that satisfies both the ``if'' and ``only if'' parts, while guaranteeing capture.

To this end, at any time $t > 0$, suppose the history $\mathcal{H}(t)$ satisfies one of the two conditions in the proposition.  Then the pursuer immediately knows which goal node the evader is heading towards.  By the triangle inequality, if the pursuer is guaranteed to get to that goal node before the evader by following the policy $\mu^*$, it is also guaranteed to do so by going directly to the goal node.  Thus, define the policy $\mu'$ to be the same as $\mu^*$ when $\mathcal{H}(t)$ does not satisfy either condition in the proposition, and to have the pursuer go directly to the appropriate goal node when $\mathcal{H}(t)$ satisfies one of the conditions in the proposition.  The policy $\mu'$ guarantees capture and satisfies the ``if'' part of the proposition.

We will now further modify $\mu'$ to obtain a policy $\mu$ that satisfies both  the ``if'' and ``only if'' parts.  Suppose that at some $t > 0$, the  history $\mathcal{H}(t)$ does not satisfy one of the two conditions in the proposition, but that $\mu'(\mathcal{H}(t)) \in \mathcal{G}$.
Then there are at least two different goal nodes that are possible targets for the evader.  Since the evader arrival time at both nodes is equal to $L$, and since $\mu'$ guarantees capture, the policy must cause the pursuer to return to a core node to resolve the ambiguity before time $L$.  Otherwise, the evader can escape through a goal node that is not being occupied by the pursuer at time $L$.   Furthermore, since the pursuer has to return to a core node before time $L$, it gains no new information about the evader's path by visiting the goal node prescribed by $\mu'$.  Thus, by having the pursuer go directly to the core node that is eventually visited under $\mu'$, and by the triangle inequality, the pursuer is still guaranteed capture of the evader.  Therefore, define the policy $\mu$ to be the same as $\mu'$, except substitute goal nodes with the eventually visited core nodes for histories that do not satisfy the conditions in the proposition.  This policy $\mu$ guarantees capture, and only visits a goal node after all ambiguity about the evader's path has been resolved.
\end{proof}

The above result indicates that for spider networks, if it is possible for the pursuer to guarantee capture, it can do so by visiting each of the core nodes in some sequence until it finds a core node in a red state, or visits $r-1$ core nodes in green states.  In the former case, the pursuer then visits the goal node corresponding to the core node in the red state.  In the latter case, the pursuer visits the goal node corresponding to the unvisited core node.   In both cases, the pursuer captures the evader at the corresponding goal node.

We are now in a position to characterize the complexity of determining whether there is a policy that guarantees capture with a delay $D > 0$.  

\section{NP-Hardness of the Capture Feasibility and Maximum Pursuer Delay Problems}
\label{sec:complexity}

To show that the Capture Feasibility Problem (i.e., Problem~\ref{prob:capture_feasibility}) is NP-hard, we will give a reduction from the NP-hard Traveling Salesperson Problem, defined as follows \cite{papadimitriou1998combinatorial}.

\begin{definition}
An instance of the metric Traveling Salesperson Problem (TSP) consists of an undirected complete graph $\mathcal{J} =\{\mathcal{V}_J, \mathcal{E}_J\}$ with $n$ nodes, and a distance function $d: \mathcal{V}_J \times \mathcal{V}_J \rightarrow \mathbb{Z}_{\ge 0}$ satisfying the triangle inequality.
\label{def:TSP_instance}
\end{definition}

\begin{problem}
{\it Traveling Salesperson Problem (Decision Version).} 
Given an instance of the metric TSP along with a positive integer $T$, does there exist a cycle (tour) that visits all nodes in the graph and has total length strictly less than $T$?
\label{prob:TSP-DEC}
\end{problem}

We now provide the following theorem characterizing the complexity of the Capture Feasibility Problem.   

\begin{theorem}
The Capture Feasibility Problem (CFP) is NP-hard.
\label{thm:CFP_NPhard}
\end{theorem}

The proof of the theorem proceeds by taking any given instance of the metric TSP and carefully constructing an instance of the CFP on a spider network.  The answer to the constructed instance of the CFP is ``yes'' if only if the answer to the given instance of the metric TSP is ``yes.'' Since the TSP is NP-hard, the CFP is NP-hard as well.    The full proof is provided in Appendix~\ref{sec:CFP_proof}.  Theorem~\ref{thm:CFP_NPhard} immediately yields the following corollary.

\begin{corollary}
The Maximum Pursuer Delay Problem (MPDP) is NP-hard.
\label{cor:MPDP_NPhard}
\end{corollary}

\begin{proof}
Given any instance of the CFP, we can answer ``yes'' or ``no'' by first solving the MPDP on that instance, and determining whether the maximum delay is zero or nonzero.  Thus, the MPDP is also NP-hard.
\end{proof}

\section{Inapproximability of the Maximum Pursuer Delay Problem}
\label{sec:approx}

A typical approach to deal with NP-hard problems is to seek {\it approximation algorithms} that yield solutions within a guaranteed constant factor of the optimal \cite{williamson2011design}.  The approximation factor for such algorithms is defined as follows.

\begin{definition}
Suppose $\mathcal{P}$ is a maximization problem.  For $\alpha \ge 1$, an algorithm $\Gamma$ is said to be an $\alpha$-approximation algorithm for $\mathcal{P}$ if for every instance of $\mathcal{P}$, the solution $D$ provided by $\Gamma$ satisfies $D\le D^* \le \alpha D$, where $D^*$ is the optimal solution.
\end{definition}

Here, we show the following negative result for the MPDP. 

\begin{theorem}
It is NP-hard to approximate the solution to the MPDP within any constant finite factor.
\end{theorem}

\begin{proof}
Suppose there exists an approximation algorithm $\Gamma$ for the MPDP that yields a constant approximation factor $\alpha \ge 1$.  We claim that $\Gamma$ solves the CFP.  Specifically, for the given instance of Pursuit-Evasion, run $\Gamma$ on the instance.  Then, $\Gamma$ will return a positive solution if and only if capture is feasible with a positive delay (since $\alpha \ge 1$).  Since the CFP is NP-hard, approximating the MPDP to within any constant finite factor is NP-hard as well.
\end{proof}

\section{Computing Maximum Pursuer Delay for a Class of Pursuit Policies}
\label{sec:sweeping_policy}

In this section, we analyze a specific class of pursuer policies on tree networks and provide an algorithm to calculate the maximum pursuer delay for such policies and networks.  The main idea behind this class of policies is that the pursuer works its way down the tree, examining the children of a given node until it isolates the subtree that the evader took, and then focusing on that subtree.  We start with some formal definitions.

\subsection{Node-Sweeping Policy}

\begin{definition}
Let $\mathcal{T}= \{\mathcal{V}, \mathcal{E}\}$ be a directed tree, rooted at a node $s \in \mathcal{V}$.  For any given node $v \in \mathcal{V}$, the {\it depth} of that node is the number of edges in the unique path from $s$ to $v$ in the tree.  The depth of node $s$ is taken to be $0$.
\label{def:depth}
\end{definition}

\begin{definition}
Consider an instance of the Pursuit-Evasion problem, where the road network $\mathcal{R}$ is a tree. For each node $v \in \mathcal{V}_{\mathcal{R}}$, define $\mathcal{L}(v)$ to be the descendant of $v$ (or $v$ itself) with the largest depth, such that if the pursuer passes through $v$, it is also guaranteed to pass through $\mathcal{L}(v)$. 
\label{def:lowest_known_descendant}
\end{definition}

To illustrate the above concept, consider the road network in Fig.~\ref{fig:illustration}. If the evader passes through node $v_3$, it is also guaranteed to pass through nodes $v_6$ and $v_8$.  In this case $\mathcal{L}(v_3) = v_8$.  Similarly, $\mathcal{L}(v_6) = v_8$, $\mathcal{L}(v_8) = v_8$, $\mathcal{L}(v_2) = v_5$, and $\mathcal{L}(v_4) = v_7$.

\begin{figure}[t]
\begin{center}
\begin{tikzpicture}[scale = 0.25]
\def\nodesize{20pt}

  \node [circle, inner sep=0pt, minimum size=\nodesize, draw] (s) at (0,0)  {\footnotesize $s$};

   \node [circle, inner sep=0pt, minimum size=\nodesize, draw] (v2) at (4,-4)  {\footnotesize $v_{2}$};
   \node [circle, inner sep=0pt, minimum size=\nodesize, draw] (v3) at (5,0)  {\footnotesize $v_{3}$};
   \node [circle, inner sep=0pt, minimum size=\nodesize, draw] (v4) at (4,4)  {\footnotesize $v_{4}$};
   
   \node [circle, inner sep=0pt, minimum size=\nodesize, draw] (v5) at (8,-4)  {\footnotesize $v_{5}$};
   \node [circle, inner sep=0pt, minimum size=\nodesize, draw] (v6) at (10,0)  {\footnotesize $v_{6}$};
   \node [circle, inner sep=0pt, minimum size=\nodesize, draw] (v7) at (8,4)  {\footnotesize $v_{7}$};
   
   \node [circle, inner sep=0pt, minimum size=\nodesize, draw] (v8) at (15,0)  {\footnotesize $v_{8}$};

   \node [circle, inner sep=0pt, minimum size=\nodesize, draw] (v9) at (18,-3)  {\footnotesize $v_{9}$};   
   \node [circle, inner sep=0pt, minimum size=\nodesize, draw] (v10) at (18,3)  {\footnotesize $v_{10}$};

  \foreach \from/\to in {s/v2,s/v3,s/v4,v2/v5,v3/v6,v4/v7,v6/v8,v8/v9,v8/v10}
      \draw [thick, ->] (\from) -> (\to);
 

\end{tikzpicture}
\caption{A tree network that illustrates various features of node-sweeping policies.}
\label{fig:illustration}
\end{center}
\end{figure}
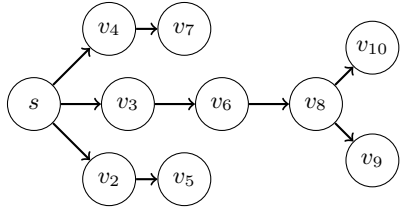

\begin{definition}[Node-Sweeping Policies]
 Let $\Pi$ be the set of pursuer policies that satisfy the following condition: the pursuer visits a node at depth $k \in \mathbb{N}$ only if it has identified the eventual state (green or red) of all nodes at depths  $\{0, 1, \ldots, k-1\}$.  We refer to the elements of $\Pi$ as {\it node-sweeping policies}.  \label{def:sweeping_policies}
\end{definition}

\begin{definition}[Node-Sweep]
Let $\{v_{i_1}, v_{i_2}, \ldots, v_{i_\mu}\}$ be the set of children of a given node.  A {\it node-sweep} is a permutation $\{v_{j_1}, v_{j_2}, \ldots, v_{j_{\mu}}\}$ of the nodes  such that the pursuer visits each of the nodes $v_{j_1}, v_{j_2}, \ldots, v_{j_{\mu-1}}$ in sequence to characterize the final state (red or green) of each visited node.  If all of the visited nodes are in a green state, the pursuer can end at any child of $\mathcal{L}(v_{j_{\mu}})$, or at a node on the path between $v_{j_{\mu}}$ and $\mathcal{L}(v_{j_{\mu}})$.   
\label{def:sweep}
\end{definition}

To illustrate the concept of a node-sweep, consider Fig.~\ref{fig:illustration} again. One possible node-sweep of the source's children is given by the permutation $\{v_2, v_4, v_3\}$. For this node-sweep, the pursuer visits the node $v_2$, waits until it is sure the evader is not going through that node, visits node $v_4$, and waits until it is sure the evader is not going through that node.  If both $v_2$ and $v_4$ are confirmed to be in a green state, the pursuer knows that the evader will go through node $v_3$.  By the definition of $\mathcal{L}(v_3)$, the pursuer also knows that $v_6$ and $v_8$ will eventually be in a red state.  By the definition of a node-sweeping policy, the pursuer is now allowed to visit $v_9$ or $v_{10}$, since it has identified the eventual states of all nodes at lower depths.  It can also visit any of $v_3, v_6, v_8$ if it so chooses.


\subsection{Complexity of Finding Optimal Node-Sweeping Policies}
Note that for a given network, there may be many strategies for sweeping the children of a given node (corresponding to the various permutations of those children).  In fact, optimally choosing the sweep sequence to solve the CFP or MPDP will be difficult in general, as indicated by the following corollary to Theorem~\ref{thm:CFP_NPhard}.

\begin{corollary}
The Capture Feasibility Problem and the Maximum Pursuer Delay Problem are NP-hard, even when restricted to node-sweeping policies on tree networks.
\end{corollary}

\begin{proof}
The proof follows from Theorem~\ref{thm:CFP_NPhard} and Corollary~\ref{cor:MPDP_NPhard} by noting that node-sweeping policies are, in fact, optimal for the spider networks defined in Section~\ref{sec:spider_network}. 
\end{proof}


\subsection{Calculating Maximum Pursuer Delay Under Node-Sweeping Policies in Tree Networks}

We now provide an algorithm to calculate the maximum pursuer delay for node-sweeping policies in general tree networks.  While this algorithm will not run in polynomial-time in general trees (due to the NP-hardness of the problem), we show that it will run in linear time (in the number of nodes) when the maximum out-degree of the tree (i.e., the number of children of any node) is bounded.  We will use the following definitions.

\begin{definition}[Evader Arrival Time at Node]
Given a tree network $\mathcal{R} = \{\mathcal{V}_{\mathcal{R}}, \mathcal{E}_{\mathcal{R}}\}$, for each node $v \in \mathcal{V}_{\mathcal{R}}$, let $t(v)$ be the evader distance from the source node $s$ to node $v$ (i.e., it is the time at which the evader would pass through node $v$ if its path goes through that node).
\end{definition}

\begin{definition}[Latest Pursuer Arrival Time at Node]
For each $v \in \mathcal{V}_{\mathcal{R}}$, let $D(v)$ be the {\it latest} time at which the pursuer can arrive at node $v$ and still be guaranteed to capture  the evader via a node-sweeping policy, {\it given} that the evader's path goes through node $v$.  
\end{definition}

\begin{definition}[Latest Time to Begin Sweep at Node]
For each $v \in \mathcal{V}_{\mathcal{R}}\setminus\{s\}$, let $S(v)$ be the {\it latest} time that the pursuer can {\it begin} a node-sweep at node $v$ (i.e., visit the siblings of node $v$, starting at $v$) and still be guaranteed to eventually capture the evader, {\it given} that the evader went through the parent of node $v$.  We define $S(v) = -\infty$ if it is not possible to guarantee (eventual) capture via a node-sweep starting at $v$.  If $v$ has no siblings, then $S(v) = D(v)$.
\label{def:latest_time_to_begin_sweep}
\end{definition}

We have the following relationships between the three terms defined above.

\begin{lemma}
For each node $v \in \mathcal{V}_{\mathcal{R}}$, $D(v) \ge t(v)$.  Furthermore, if $S(v) \ne -\infty$, $D(v) \ge S(v) \ge t(v)$.
\label{lem:delay_sweep_arrival}
\end{lemma}

\begin{proof}
Since capture is guaranteed if the pursuer arrives at node $v$ at time $t(v)$ (given that the evader goes through that node), we have $D(v) \ge t(v)$. Now, suppose that $S(v) \ne -\infty$, which means that the pursuer can begin a node-sweep at node $v$ and still be guaranteed to capture the evader via a node-sweeping policy.  If $S(v) > D(v)$ and the evader has gone through $v$, then it is impossible for the pursuer to capture the evader (by definition of $D(v)$).  Thus, $S(v) \le D(v)$.  Furthermore, by the definition of a node-sweep, the pursuer must conclusively determine the eventual state (red or green) of each node that it visits in the sweep.  Thus, the pursuer must depart node $v$ no earlier than $t(v)$ (since that is when the evader would get to that node).  Since the pursuer can depart a node as soon as it arrives, the latest time to begin a node-sweep at node $v$ must satisfy $S(v) \ge t(v)$, proving the claim.
\end{proof}

We now describe how to calculate the specific values of $D(\cdot)$ and $S(\cdot)$ for the nodes in the network.

\subsection*{Calculating $D(\cdot)$}
The following result immediately follows from the definitions of $D(v)$ and $t(v)$.

\begin{lemma}
For every goal node $v \in \mathcal{G}$, $D(v) = t(v)$.  
\label{lem:delay_goal_nodes}
\end{lemma}

For each non-goal node in the network, the following result describes how to calculate $D(\cdot)$ in terms of the functions $t(\cdot)$ and $S(\cdot)$ for the descendants of that node.

\begin{lemma}
Consider any non-goal node $v \in \mathcal{V}_{\mathcal{R}}$, and the corresponding node $\mathcal{L}(v)$ defined in Definition~\ref{def:lowest_known_descendant}. Let $C(\mathcal{L}(v))$ be the children (if any) of $\mathcal{L}(v)$.  If $C(\mathcal{L}(v))$ is empty or $S(w) = -\infty$ for all $w \in C(\mathcal{L}(v))$, then $D(v) = t(\mathcal{L}(v)) - d_p(v, \mathcal{L}(v))$. Otherwise, 
$$
D(v) = \max_{w\in C(\mathcal{L}(v))}\left\{S(w) - d_p(v,w)\right\}.
$$
\label{lem:D_for_all_nodes}
\end{lemma}

\begin{proof}
By Lemma~\ref{lem:delay_sweep_arrival}, $D(v) \ge t(v)$.  Suppose the pursuer arrives at $v$ at some time strictly larger than $t(v)$, and that the evader has passed through $v$ (so that the node is in a red state).  Then the pursuer knows that all nodes on the path from $v$ to $\mathcal{L}(v)$ will (eventually) be in a red state, and can thus visit any of those nodes or the children of $\mathcal{L}(v)$ under a node-sweeping policy.  

Consider node $\mathcal{L}(v)$.  If this is a goal node, then $D(\mathcal{L}(v)) = t(\mathcal{L}(v))$ by Lemma~\ref{lem:delay_goal_nodes}.  If this is not a goal node, and the pursuer arrives at $\mathcal{L}(v)$ after time $t(\mathcal{L}(v))$, then the pursuer must be prepared to sweep its children.  If it chooses a child $w$ of $\mathcal{L}(v)$ to start its sweep, it must get to $w$ no later than $S(w)$.  Thus, the latest the pursuer can get to $\mathcal{L}(v)$ is the maximum (over all children of $\mathcal{L}(v)$) of the quantity $S(w)-d_p(\mathcal{L}(v),w)$.  If it is not possible to guarantee capture via a node-sweep starting at any child (i.e., $S(w) = -\infty$ for all $w \in C(\mathcal{L}(v))$), then the pursuer must capture the evader before it passes $\mathcal{L}(v)$, and thus $D(\mathcal{L}(v)) = t(\mathcal{L}(v))$.  Thus, the value of $D(\mathcal{L}(v))$ is as indicated by the lemma.

If $\mathcal{L}(v) \ne v$, consider the node $z$ that directly precedes $\mathcal{L}(v)$ on the path from $v$ to $\mathcal{L}(v)$.  If the pursuer arrives at $z$ after time $t(z)$, it can either choose to proceed to $\mathcal{L}(v)$ or to a child $w$ of $\mathcal{L}(v)$.  In the latter case, as above, the latest the pursuer can arrive at $z$ is the maximum value (over all children of $\mathcal{L}(v)$) of the quantity $S(w)-d_p(z,w)$.  If the pursuer instead proceeds to $\mathcal{L}(v)$, it must arrive there before time $D(\mathcal{L}(v))$.  In this case, the pursuer must get to $z$ no later than $D(\mathcal{L}(v)) - d_p(z, \mathcal{L}(v))$.  Recall from the preceding argument that if there is some sweep of $\mathcal{L}(v)$'s children that guarantees capture, then $D(\mathcal{L}(v))$ is given by $S(w)-d_p(\mathcal{L}(v),w)$ for some child $w$.  Thus, if the pursuer chooses to go to $\mathcal{L}(v)$ from $z$, it must get to $z$ no later than
\begin{align*}
D(\mathcal{L}(v)) - d_p(z,\mathcal{L}(v)) &= S(w)-d_p(\mathcal{L}(v),w) - d_p(z,\mathcal{L}(v)) \\
&\le S(w) - d_p(z,w)
\end{align*}
by the triangle inequality.  In other words, by proceeding directly to the child node $w$ from $z$ (rather than first visiting $\mathcal{L}(v)$), the pursuer can get to $z$ later and still guarantee capture.  Thus, starting at $z$, the pursuer should directly go to a child of $\mathcal{L}(v)$ to begin a node-sweep (if that option is available), or else, go to $\mathcal{L}(v)$ to capture the evader. This yields the value $D(z)$, as indicated by the lemma.

The above argument can be repeated by working backwards for each node on the path from $v$ to $\mathcal{L}(v)$.  In each case, it will be optimal (by the triangle inequality) for the pursuer to go directly to a child of $\mathcal{L}(v)$ to begin a node-sweep (if that option is available), or else go to $\mathcal{L}(v)$ to capture the evader.
\end{proof}

\subsection*{Calculating $S(\cdot)$}

To characterize $S(\cdot)$, we first define the following notion.

\begin{definition}
Consider a node $v \in \mathcal{V}_{\mathcal{R}}$ and its children $C(v) = \{v_{i_1}, v_{i_2}, \ldots, v_{i_\mu}\}$.\footnote{We omit the dependence of $\mu$ on $v$ for notational convenience.}  Consider a node-sweep $P$ of the children, with the associated node sequence $\{v_{j_1}, v_{j_2}, \ldots, v_{j_{\mu-1}}, w\}$, where $w$ is either a child of $\mathcal{L}(v_{j_{\mu}})$ or some node on the path between $v_{j_{\mu}}$ and $\mathcal{L}(v_{j_{\mu}})$.   We say that $P$ is a {\it feasible node-sweep starting at $v_{j_1}$} if it is possible to guarantee (eventual) capture of the evader after following that sweep.  If $P$ is a feasible node-sweep, the {\it latest start time for node-sweep $P$} is the latest time that the pursuer can arrive at (and depart) node $v_{j_1}$ and still maintain feasibility of the node-sweep.
\label{def:feasible_sweep}
\end{definition}

The following result characterizes the feasibility of a node-sweep starting at a given node, in terms of the quantities $t(\cdot), D(\cdot)$ and $S(\cdot)$.

\begin{lemma}
Consider a set of siblings $\{v_{i_1}, v_{i_2}, \ldots, v_{i_\mu}\}$, and a node-sweep $P$ with node sequence $\{v_{j_1}, v_{j_2}, \ldots, v_{j_{\mu-1}}, w\}$, where $w$ is either a child of $\mathcal{L}(v_{j_{\mu}})$ or some node on the path between $v_{j_{\mu}}$ and $\mathcal{L}(v_{j_{\mu}})$.  Then $P$ is a feasible node-sweep if and only if:
\begin{itemize}
\item The pursuer can follow the sequence  $v_{j_1}, v_{j_2}, \ldots, v_{j_{\mu-1}}$ such that it visits each node $v_{j_k}$ in the interval $[t(v_{j_k}), D(v_{j_k})]$, $k \in \{1, 2, \ldots, \mu-1\}$.
\item If $w$ is a node on the path between $v_{j_{\mu}}$ and $\mathcal{L}(v_{j_{\mu}})$, the pursuer arrives at $w$ before $D(w)$.
\item If $w$ is a child of $\mathcal{L}(v_{j_{\mu}})$, the pursuer arrives at $w$ before $S(w)$.
\end{itemize}
\label{lem:S_characterization}
\end{lemma}

\begin{proof}
By the definition of a node-sweep (Definition~\ref{def:sweep}), the pursuer must conclusively establish the eventual state of each node that it visits in the sweep.  Thus, for $1 \le j \le \mu-1$, the pursuer must visit node $v_{j_k}$ after time $t(v_{j_k})$, since that is the earliest time at which the evader will visit that node.  For a feasible policy, if any of the visited nodes is in a red state, the pursuer must be able to guarantee capture via a node-sweeping policy starting at that node.  Thus, for each $1 \le j \le \mu-1$, the pursuer must reach node $v_{j_k}$ no later than $D(v_{j_k})$.  This establishes the first item in the result.

Now suppose the first $\mu-1$ nodes are in a green state.  If the last node $w$ in the node-sweep is on the path between $v_{j_{\mu}}$ and $\mathcal{L}(v_{j_{\mu}})$, the pursuer knows conclusively that the evader will go through that node (or has already done so).  In this case, the pursuer must arrive at that node in time to guarantee (eventual) capture, leading to the second item in the result.  If the last node $w$ is a child node of $\mathcal{L}(w)$, then the (eventual) state of that node is undetermined before the pursuer gets there.   Thus, the pursuer must arrive at that node in time to begin a node-sweep, leading to the third item in the result.
\end{proof}

The above results show that evaluating feasibility of a node-sweep boils down to solving a  vehicle routing problem with time-windows \cite{bansal2004approximation}. In Appendix~\ref{sec:latest_departure_time}, we provide a linear-time algorithm to find the latest start time for a {\it given} node-sweep in order to meet the time-window constraints identified in the above lemma (the left endpoint of the time-window for the last node in the sequence is taken to be zero).  The algorithm returns $-\infty$ if it is not possible to meet the time-window constraints (i.e., if the given node-sweep is not feasible).

Based on the above definition and characterization of feasible node-sweeps, the following characterization of $S(v)$ follows by definition.

\begin{lemma}
For each $v \in \mathcal{V}_{\mathcal{R}} \setminus \{s\}$, let $\mathcal{P}_v$ be the set of all feasible node-sweeps starting at $v$.  Then $S(v)$ is the maximum of the latest start times over all node-sweeps in $\mathcal{P}_v$.  If $v$ has no siblings, then $S(v) = D(v)$.  If there are no feasible node-sweeps that start at $v$, $S(v) = -\infty$.
\label{lem:S_for_all_nodes}
\end{lemma}

The above characterization of $S(v)$ relies on the notion of a feasible node-sweep, which is defined in terms of the values of $S(\cdot)$ for descendants of siblings of $v$ (from Lemma~\ref{lem:S_characterization}).  However, when a given goal node has only other goal nodes as siblings, we can directly calculate the value of $S(\cdot)$ for such nodes, as given by the following result.  Thus, such nodes will serve as ``base cases'' to calculate the values of $S(\cdot)$ (and $D(\cdot)$ via Lemma~\ref{lem:D_for_all_nodes}) for the other nodes in the network.

\begin{lemma}
Consider a node $v \in \mathcal{V}_{\mathcal{R}}$ such that all of its children $C(v) = \{v_{i_1}, v_{i_2}, \ldots, v_{i_\mu}\}$ are goal nodes.  Assume that the children are ordered such that $t(v_{i_k}) \le t(v_{i_{k+1}})$ for all $1 \le k \le \mu-1$.  Then, if the evader passes through $v$, capture is guaranteed if and only if 
\begin{equation}
d_p(v_{i_{k}}, v_{i_{k+1}}) \le t(v_{i_{k+1}}) - t(v_{i_k}), \enspace \forall k \in \{1, \ldots, \mu-1\}.
\label{eq:goal_node_sweep}
\end{equation}
If this condition is satisfied, then $S(v_{i_k}) = D(v_{i_k})~(= t(v_{i_k}))$ for all $k$ such that $t(v_{i_k}) = t(v_{i_1})$, and $S(v_{i_k}) = -\infty$ for all other $k$.  On the other hand, if the above capture condition is not satisfied, then $S(w) = -\infty$ for all $w \in C(v)$.
\label{lem:S_for_goal_nodes}
\end{lemma}

\begin{proof}
Consider the node $v$ and its children $C(v)$.  If the evader has passed through $v$, then it can only be captured at one of the nodes in $C(v)$.  This is only possible if the purser can visit each of the nodes in $C(v)$ in sequence of increasing evader arrival times, departing each node $w \in C(v)$ no earlier than time $t(w)$.  If this condition is not satisfied, there is no sweep of the nodes in $C(v)$ that is guaranteed to capture the evader.  On the other hand, if the condition is satisfied, a feasible node-sweep can only begin at the goal node with the earliest evader arrival time.  If there are multiple nodes with arrival time equal to $t(v_{i_1})$, then the satisfaction of \eqref{eq:goal_node_sweep} implies that the pursuer has zero travel time between those nodes, and can begin a node-sweep at any of those nodes.  This leads to the stated result.
\end{proof}

\subsection*{Summary: An Algorithm to Calculate $D(\cdot)$ and $S(\cdot)$ for All Nodes}

Lemma~\ref{lem:delay_sweep_arrival} provides $D(v)$ for all goal nodes, and Lemma~\ref{lem:S_for_goal_nodes} provides $S(v)$ for all goal nodes that only have other goal nodes as siblings.  Using these quantities as base cases, Lemma~\ref{lem:D_for_all_nodes} and Lemma~\ref{lem:S_for_all_nodes} can be iteratively applied to calculate $D(\cdot)$ and $S(\cdot)$ for all of the other nodes in the network.  Algorithm~\ref{alg:MPD_sweeping_tree} provides the means to do this.

\begin{algorithm}
	\caption{Find Maximum Pursuer Delay for Node-Sweeping Policies on Tree Networks}
	\label{alg:MPD_sweeping_tree}
	\textbf{Input}: An instance of the Pursuit-Evasion problem, where the road network is a tree and contains $n$ nodes.  The nodes are assumed to be sorted according to a topological ordering, where $s$ is the first node and each node has outgoing edges only to nodes later in the ordering.\\
	\textbf{Output}: The latest time at which the pursuer can arrive at the source node $s$ and be guaranteed to capture the evader via a node-sweeping policy.
	\begin{algorithmic}[1]
		\STATE{For each $v \in \mathcal{V}_{\mathcal{R}}$, calculate evader arrival time $t(v)$.}
		\STATE{For each goal node $v$, set $D(v) = t(v)$.}
		\STATE{For each non-goal node $v$ whose children $C(v)$ are all goal nodes, calculate $S(w)$ for $w \in C(v)$ according to Lemma~\ref{lem:S_for_goal_nodes}.}
		\STATE{Let $K$ be the maximum depth of the tree.}
		
		\FOR{$i$ from $K-1$ to $1$}
		   \STATE{For each non-goal node $v$ at depth $i$, calculate $D(v)$ using Lemma~\ref{lem:D_for_all_nodes}.}
		   \STATE{For each node $v$ at depth $i$, calculate $S(v)$ according to Lemma~\ref{lem:S_characterization} and Lemma~\ref{lem:S_for_all_nodes}.}
		\ENDFOR

		\STATE{Calculate $D(s)$ for the source via Lemma~\ref{lem:D_for_all_nodes} and return that value.}
		
	\end{algorithmic}
\end{algorithm}

Based on Algorithm~\ref{alg:MPD_sweeping_tree}, we have the following result.

\begin{theorem}
For any given tree network, let $\Delta$ be the largest out-degree.  It is possible to calculate the maximum pursuer delay over all node-sweeping policies in $O((\Delta+1)!n)$-time, where $n$ is the number of nodes in the tree.  Thus, in bounded-degree trees, the maximum pursuer delay for node-sweeping policies can be calculated in linear time.
\end{theorem}

\begin{proof}
We will characterize the time-complexity of each step in Algorithm~\ref{alg:MPD_sweeping_tree}.  The calculation of the evader travel times $t(\cdot)$ in step $1$ can be done in $O(n)$-time by applying a simple Breadth First Search (BFS) algorithm on the tree \cite{cormen2009introduction}.  Step $2$ also takes $O(n)$ time. Step $3$ requires checking each node to see if its children are all goal nodes, and if so, determining whether the condition in Lemma~\ref{lem:S_for_goal_nodes} is satisfied.  This takes $O\left(n\Delta\ln(\Delta)\right)$ time (where the $\Delta\ln(\Delta)$ term corresponds to sorting the children of a given node by their $t(\cdot)$ values).  Finding the maximum depth of the tree in Step $4$ can be done in $O(n)$-time via BFS.  The iteration in lines $5$-$8$ visits each node in the network at most once.  For each such node $v$, calculating $D(v)$ via Lemma~\ref{lem:D_for_all_nodes} requires $O(\Delta)$ time.  Calculating $S(v)$  via Lemma~\ref{lem:S_characterization} and Lemma~\ref{lem:S_for_all_nodes} requires us to check (in the worst case) all possible node-sweeps starting at $v$ and determining feasibility.  There are $O((\Delta-1)!)$ possible permutations of the siblings of $v$.  For each permutation, the node-sweep can end at either $\mathcal{L}(v_{j_\mu})$ or at one of the children of $\mathcal{L}(v_{j_{\mu}})$, where $v_{j_{\mu}}$ is the last node in the permutation.  Thus, there are $O(\Delta!)$ possible sweeps to check.  Evaluating feasibility of a particular sweep of $\Delta$ nodes (and calculating the latest start time for the sweep) can be done in time $O(\Delta)$ via Algorithm~\ref{alg:latest_departure} provided in Appendix~\ref{sec:latest_departure_time}.  Thus, calculating $S(v)$ via line $7$ takes time $O\left(\Delta\Delta!\right) = O\left((\Delta+1)!\right)$.  Finally, calculating $D(s)$ in line $9$ takes time $O(\Delta)$.  Thus, the overall time-complexity of the algorithm is $O\left((\Delta+1)!n\right)$.
\end{proof}

\section{Numerical Evaluation}
To evaluate the potential of node-sweeping policies, we compared the maximum pursuer delay for such policies against the maximum pursuer delay over all policies (provided by the algorithm from \cite{krishnamoorthy2016pursuit}) on a real road network at Camp Atterbury, Indiana, USA.  The tree corresponding to the road network consists of $87$ nodes with maximum out degree of $3$.  The evader's speed was set to $20$ mph, and the pursuer's speed was set to $40$ mph.  The maximum pursuer delay for the optimal node-sweeping policy was calculated (using Algorithm~\ref{alg:MPD_sweeping_tree}) to be $7.841$ minutes, while the maximum pursuer delay over all policies was found (using the optimal algorithm from \cite{krishnamoorthy2016pursuit}) to be $7.897$ minutes.  However, the time taken to calculate the optimal node-sweeping policy was $0.17$ seconds on a laptop, whereas the time taken to calculate the maximum delay over all policies took more than $4$ hours;  thus, for this road network, node-sweeping policies provide less than $1$\% loss in performance, while requiring a computation time that is several orders of magnitude smaller than the optimal policy.
\section{Conclusions and Future Work}
In this paper, we studied the problem of capturing an evader in a graph, where the pursuer only obtains information about the evader's path via sensors located at the nodes.  We showed that it is NP-hard to find optimal policies to capture the evader.  Nevertheless, for a certain class of policies, we provided a linear-time algorithm to calculate the maximum pursuer delay in bounded-degree tree networks.  There are many interesting avenues for future work, including generalizing  node-sweeping policies to handle uncertain evader travel times and multiple pursuers or evaders.  It would also be of interest to identify classes of networks where node-sweeping policies are optimal.


\sloppy
\bibliographystyle{plainnat}
{\footnotesize \bibliography{refs}}

\begin{appendices}

\section{Proof of Theorem~\ref{thm:CFP_NPhard}}
\label{sec:CFP_proof}

\begin{proof}
Consider an instance of the decision version of the TSP, with complete graph $\mathcal{J} = \{\mathcal{V}_J, \mathcal{E}_J\}$, associated distance function $d(\cdot, \cdot)$, and target tour length $T$.  Let the number of nodes in the graph be denoted by $n$.  We will assume without loss of generality that $T > 0$, as the answer to the given instance of the TSP is trivially ``no'' otherwise.

We will first transform the distances in the given instance in such a way that the answer (yes or no) to the instance does not change, but that will allow us to perform a reduction to the CFP.\footnote{In particular, this transformation will increase the length of all edges in the TSP by a certain amount; this will allow us to construct an instance of the CFP that meets the constraint that the pursuer's travel times are smaller than those of the evader.}   To do this, let $c$ be defined as
\begin{equation}
c \triangleq \left\lceil \frac{2\max_{v, w \in \mathcal{V}_J}d(v, w)}{n-2} \right\rceil.
\label{eq:TSP_padding}
\end{equation}
Now define the new distances $d'(v, w) \triangleq d(v, w) + c$ for all $v \ne w$, and $T' \triangleq T+nc$.  Thus, the original instance of the TSP (with distances given by $d(\cdot, \cdot)$) will have a tour of length less than $T$ if and only if the modified instance (with distances given by $d'(\cdot, \cdot)$) has a tour of length less than $T'$.  We will therefore work with the modified instance in the rest of the proof.   Given the above (modified) instance of the TSP, we create an instance of the CFP as follows.

Pick any node $s \in \mathcal{V}_J$ to be the source node, and define the core nodes to be $\mathcal{C} \triangleq \mathcal{V}_J \setminus \{s\}$.  Create one goal node for every core node in $\mathcal{C}$, and denote the set of goal nodes by $\mathcal{G}$.  Denote $\mathcal{V}_R \triangleq \{s\} \cup \mathcal{C} \cup \mathcal{G}$, and choose the edge set $\mathcal{E}_R$ to create a spider network on $\mathcal{V}_R$, as described in Section~\ref{sec:spider_network}.  

We now specify the pursuer distances on the constructed network.  Let the pursuer distances between nodes in $\{s\} \cup \mathcal{C}$ (i.e., $\mathcal{V}_J$) be the same as in the given instance of the TSP, namely $
d_p(v, w) \triangleq d'(v, w)$ $\forall v, w \in \mathcal{V}_J$. Define the pursuer distances between the other nodes in the graph as follows.
\begin{itemize}
\item (Core nodes to their corresponding goal nodes):   For each $1 \le i \le n-1$, define $d_p(v_{c_i}, v_{g_i}) \triangleq d_p(v_{c_i},s)$, 
i.e., the time it takes the pursuer to get to $v_{g_i}$ from $v_{c_i}$ is the same as the time to go to $s$ from $v_{c_i}$. 
\item (Core nodes to other goal nodes): For each $1 \le i \le n-1$, for each $1 \le j \le n-1$ such that $j \ne i$, define $d_p(v_{c_i}, v_{g_j}) \triangleq d_p(v_{c_i},v_{c_j}) + d_p(v_{c_j}, v_{g_j})$, i.e., the shortest distance between a core node $v_{c_i}$ and a goal node $v_{g_j}$ for $j \ne i$ is to first go to the core node $v_{c_j}$ and then go to the goal node $v_{g_j}$.
\item (Source node to goal nodes): For each $1 \le i \le n-1$, define $d_p(s,v_{g_i}) \triangleq d_p(s, v_{c_i}) + d_p(v_{c_i}, v_{g_i})$, i.e., the shortest path from the source node to a goal node is to first go to the corresponding core node, and then travel to the goal node.
\item (Goal nodes to other goal nodes):  For each $1 \le i \le n-1$, for each $1 \le j \le n-1$ such that $j \ne i$, define $d_p(v_{g_i}, v_{g_j}) \triangleq d_p(v_{g_i}, v_{c_j}) + d_p(v_{c_j}, v_{g_j})$, i.e., the shortest path from one goal node to another is to go through the corresponding core nodes.
\end{itemize}
One can verify that the pursuer distance function $d_p(\cdot, \cdot)$ defined as above satisfies the triangle inequality (given that the distances in the given instance of the TSP do so).  

Finally, we specify travel times for the evader as follows:
\begin{align*}
 \enspace \forall v_{c_i} \in \mathcal{C}, \enspace &d_e(s,v_{c_i}) \triangleq d'(s,v_{c_i}), \\
 &d_e(v_{c_i}, v_{g_i}) \triangleq T' - d'(s,v_{c_i}),
\end{align*}
i.e., we take $L = T'$ in Fig.~\ref{fig:spider_network}, where $T'$ is the (modified) target tour length for the TSP.  Note that the evader travel times on the roads between core nodes and goal nodes satisfy
\begin{align*}
d_e(v_{c_i},v_{g_i}) &= T' - d'(s,v_{c_i}) \\
&= T + nc - d(s,v_{c_i}) - c \\
&= T + (n-2)c - d(s,v_{c_i}) + c \\
&\ge T + 2\max_{v, w\in \mathcal{V}_J}d(v, w) - d(s,v_{c_i}) + c \\
&\ge d(s,v_{c_i}) + c = d_p(v_{c_i},v_{g_i})
\end{align*}
as required for Pursuit-Evasion problems we are considering.  

This completes the construction of the instance of the CFP; note that this construction requires only a polynomial number of operations (in terms of the size of the given instance of the TSP).  We now argue that the answer to the constructed instance of the CFP is ``yes'' if and only if the answer to the (modified) instance of the TSP is ``yes''.

First suppose that the (modified) instance of the TSP has a tour of length $\bar{T} < T'$ (i.e., the answer to the given instance is ``yes'').  Let the pursuer enter at time $T'-\bar{T}$ and follow such a tour on the nodes $\{s\}\cup \mathcal{C}$, starting from node $s$.  There are two possible scenarios:
\begin{enumerate}
\item The pursuer encounters a core node $v_{c_i}$ in a red state while following this tour, and then heads to the corresponding goal node $v_{g_i}$.  Let $C_{s, v_{c_i}}$ be the length from node $s$ to $v_{c_i}$ on the tour, and recall from our construction that $d_p(v_{c_i},v_{g_i}) = d_p(v_{c_i}, s)$.  Then, the pursuer arrives at the goal node $v_{g_i}$ at time
$$
T' - \bar{T} + C_{s, v_{c_i}} + d_p(v_{c_i}, s) \le T',
$$
by the triangle inequality.  
\item The first $n-2$ core nodes encountered by the pursuer are in a green state, and thus the pursuer heads to the goal node corresponding to the last core node.  Let the two core nodes immediately preceding $s$ in the given TSP tour be denoted $v_{c_i}$ and $v_{c_j}$, respectively (i.e., the tour contains the sequence $v_{c_j} \rightarrow v_{c_i} \rightarrow s$).  Under this scenario, the pursuer arrives at $v_{g_i}$ at time
\begin{align*}
&T' - \bar{T} + C_{s, v_{c_j}} + d_p(v_{c_j}, v_{g_i}) \\
&= T' - \bar{T} + C_{s, v_{c_j}} + d_p(v_{c_j}, v_{c_i}) + d_p(v_{c_i}, v_{g_i}) \\
&= T' - \bar{T} + C_{s, v_{c_i}} + d_p(v_{c_i}, s) = T',
\end{align*}
by construction of the pursuer travel times.    
\end{enumerate}
In both cases, since the evader only arrives at the goal node at time $T'$, the pursuer captures the evader.  Thus, the answer to the constructed instance of the CFP is also ``yes.''   

We now prove the opposite direction, namely that if the answer to the constructed instance of the CFP is ``yes'', then the answer to the given instance of the TSP is also ``yes.''  Thus, suppose that there exists a policy $\mu$ that guarantees capture when the pursuer enters at some time $D > 0$.  By Proposition~\ref{prop:optimal_capture_goal_node}, this policy can be taken to visit the core nodes in some sequence and then head to a goal node.  In the worst case, the pursuer has to visit all but one of the core nodes (starting from node $s$) before heading to a goal node.  By construction of the network, the time taken to visit these core nodes and to travel to the last goal node is exactly equal to the length of a tour on all nodes in $\{s\} \cup \mathcal{C}$.  Let $\bar{T}$ be the length of this tour.  Thus, the pursuer arrives at time $D+\bar{T} \le T'$, which is the time at which the evader reaches the goal node.  Thus, the answer to the given instance of the TSP is also ``yes.''

All together, this shows that an algorithm for the CFP also yields an algorithm for the decision version of the TSP (which is NP-hard).  Thus, the CFP is NP-hard as well.
\end{proof}

\section{Finding the Latest Start Time to Visit a Sequence of Nodes Within Specified Time Windows}
\label{sec:latest_departure_time}

Given a sequence of nodes $v_1, v_2, \ldots, v_{\mu}$, along with travel times between nodes and a time window for each node, Algorithm~\ref{alg:latest_departure} determines the latest time at which a pursuer can leave node $v_1$ and visit all of the nodes (in sequence) within their time windows.  The algorithm maintains a vector $\tau$ of length ${\mu}$, with element $\tau[i]$ indicating the latest time that the pursuer can arrive at node $i$ in the sequence and still meet the time window constraints.  The algorithm returns $-\infty$ if it is impossible to visit the nodes in the given sequence within their time windows.  The algorithm starts by assigning the arrival time for the last node to be the end of its time window.  It then works backwards through the node sequence to find the largest time within each node's window at which the pursuer can arrive (and leave) in order to get to the next node.

\begin{algorithm}
	\caption{Find Latest Start Time For Feasible Node-Sweep}
	\label{alg:latest_departure}
	\textbf{Input}: An ordered set of vertices $v_1, v_2, \ldots, v_{\mu}$, a nonnegative distance $d_p(\cdot, \cdot)$ between each consecutive pair of vertices, a set of time windows $[t(v_i), D(v_i)]$ for each vertex $v_i$, $1 \le i \le {\mu}$, where $D(v_i) \ge t(v_i)$.\\
	\textbf{Output}: The latest time at which the pursuer can leave vertex $v_1$ and visit all of the nodes (in order) within their given time-windows.  The output is $-\infty$ if it is not possible to visit all nodes within their time-windows.
	\begin{algorithmic}[1]
		
		\STATE{Set $\tau[{\mu}] = D(v_{\mu})$} \COMMENT{The vector $\tau$ contains the latest arrival time at each node in order to meet all time-window constraints.}
		\FOR{$i$ from ${\mu}-1$ to $1$}
		   \IF{$\tau[i+1]-d_p(v_{i}, v_{i+1}) < t(v_i)$}
		      \STATE{Return $-\infty$}
		   \ELSE
		      \STATE{$\tau[i] = \min\{\tau[i+1]-d_p(v_{i}, v_{i+1}), D(v_i)\}$}
		   \ENDIF
		\ENDFOR   
		
		\STATE{Return $\tau[1]$}
		
	\end{algorithmic}
\end{algorithm}

\begin{proposition}
Algorithm~\ref{alg:latest_departure} returns $-\infty$ if and only if it is impossible to visit the sequence of nodes within their given time-windows.  Furthermore, if the algorithm returns a finite value, that value is the latest time at which the pursuer can depart node $v_1$ and visit the sequence of nodes within their time windows.
\end{proposition}

\begin{proof}
We start with the first statement, and consider the ``if'' condition.  Suppose by way of contradiction that it is impossible to visit the sequence of nodes within their time-windows, but that the algorithm returns some finite value $\tau[1]$.  Consider the vector $\tau$  created during the course of the algorithm.  From the definition of the entries in this vector, we have $\tau[i] \in [t(v_i), D(v_i)]$ for all $1 \le i \le {\mu}$.  Furthermore, $\tau[i+1] \ge \tau[i] + d_p(v_i, v_{i+1})$ for all $1 \le i \le {\mu}-1$.  Thus there is sufficient time for the pursuer to travel between each node in the sequence and arrive within the time-windows.  This contradicts the fact that such a route is impossible.

Now we prove that the algorithm returns $-\infty$ only if there is no feasible route through the sequence of nodes.  To do this, we will show the contrapositive, i.e., if there is a feasible route, the algorithm returns a finite value.  Consider some feasible route, and let  $\tau'[1], \tau'[2], \ldots, \tau'[{\mu}]$ be the sequence of arrival times at the nodes on that route.  Without loss of generality, we take $\tau'[i] \ge t(v_i)$ for all $i \in \{1, 2, \ldots, \mu\}$ since the pursuer has to be at each node $v_i$ at or after time $t(v_i)$, and can always delay its departure from the previous node to arrive exactly at that time if needed.  We now show that Algorithm~\ref{alg:latest_departure} will generate a finite sequence of arrival times.  First, note from line 1 of the algorithm that $\tau[{\mu}] = D(v_{\mu}) \ge \tau'[{\mu}]$.  Next, note that $\tau'[{\mu}-1] + d_p(v_{{\mu}-1}, v_{\mu}) \le \tau'[{\mu}] \le \tau[{\mu}]$.  Since $\tau'[{\mu}-1] \ge t(v_{{\mu}-1})$, we have $\tau[{\mu}] - d_p(v_{{\mu}-1}, v_{\mu}) \ge t(v_{{\mu}-1})$, and thus $\tau[{\mu}-1]$ gets assigned a finite value (lines 3 - 7).  In particular, by line $7$, $\tau[{\mu}-1]$ takes the largest value in the interval $[t(v_{{\mu}-1}), D(v_{{\mu}-1})]$ that allows it to reach $v_{\mu}$ by $\tau[{\mu}]$, and thus $\tau[{\mu}-1] \ge \tau'[{\mu}-1]$.  

Continuing in this way, suppose that the entries $j+1, j+2, \ldots, {\mu}$ have been assigned finite values in vector $\tau$, with $\tau[i] \ge \tau'[i]$ and $\tau[i] \in [t(v_i), D(v_i)]$ for all $j+1 \le i \le {\mu}$.  For the feasible route, we have $\tau'[j] + d_p(v_j, v_{j+1}) \le \tau'[j+1] \le \tau[j+1]$.  Since $\tau'[j] \ge t(v_j)$, we have $\tau[j+1]-d_p(v_j, v_{j+1}) \ge t(v_j)$, and thus $\tau[j]$ gets assigned a finite value. Since $\tau[j]$ takes the largest value in the interval $[t(v_j), D(v_j)]$ that allows it to reach $v_{j+1}$ by time $\tau[j+1]$, we have $\tau[j] \ge \tau'[j]$.  By induction, the algorithm generates finite values for all elements in $\tau$, and thus returns a finite value $\tau[1]$ in line 9.

The above inductive argument also shows that the value $\tau[1]$ returned by the algorithm is at least as large as the start time of any feasible route through the sequence of nodes, proving the last claim of the proposition.
\end{proof}

\end{appendices}

\end{document}